\documentclass[12pt]{article}

\usepackage{diagbox}
\usepackage{mathtools}
\usepackage{bbm}
\usepackage{latexsym}
\usepackage{epsfig}
\usepackage{amsmath,amsthm,amssymb,enumerate}

\usepackage[a-1b]{pdfx}
\usepackage{hyperref}
\newcommand{\eoc}[1]{\\{\bf End of Case #1}}
\newcommand{\upp}[1]{\langle#1\rangle}

\usepackage{color}

\def\deg{\text{deg}}

\def\cI{{\mathcal I}}
\def\cH{{\mathcal H}}

\def\e{\varepsilon}    \def\g{\gamma}
\def\G{\Gamma}  \def\k{\kappa}
 \def\th{\theta}    
  \def\n{\nu} \def\p{\pi}
\def\r{\rho}  \def\s{\sigma} 
\def\t{\tau} \def\om{\omega}

\def\cP{{\cal P}}

\def\Ic{I_{conn}}

\newtheorem{theorem}{Theorem}
\newtheorem{lemma}[theorem]{Lemma}

\newtheorem{Remark}{Remark}

\newcommand{\rdown}[1]{{\left\lfloor #1\right \rfloor}}

\newcommand{\brac}[1]{\left(#1\right)}

\newcommand{\bfrac}[2]{\left(\frac{#1}{#2}\right)}

\newcommand{\set}[1]{\left\{#1\right\}}

\def\E{\mathbb{E}}

\def\Pr{\mathbb{P}}

\newcommand{\ignore}[1]{}

\def\cA{{\mathcal A}}

\def\cG{{\mathcal G}}
\def\cH{{\mathcal H}}
\def\cI{{\mathcal I}}

\def\cK{{\mathcal K}}

\def\cM{{\mathcal M}}

\def\cP{{\mathcal P}}

\newcommand{\beq}[2]{\begin{equation}\label{#1}#2\end{equation}}
\newcommand{\mults}[1]{\begin{multline*}#1\end{multline*}}
\newcommand{\mult}[2]{\begin{multline}\label{#1}#2\end{multline}}

\def\cG{\mathcal{G}}

\def\cI{{\cal I}}

\def\bof{{\bf f}}

\def\cA{{\mathcal A}}
\usepackage{tikz}
\usetikzlibrary{decorations.pathmorphing}
\usetikzlibrary{positioning}
\usetikzlibrary{arrows,automata}
\usetikzlibrary{shapes.misc}
\usetikzlibrary{backgrounds}
\usetikzlibrary{arrows,shapes}
\def\be{{\bf e}}

\def\MT{MultiTree}
\def\MF{MultiForest}
\def\MP{MultiPerfectMatching}

\begin{document}
\author{Alan Frieze\thanks{Research supported in part by NSF grant DMS} and Wesley Pegden\thanks{Research supported in part by NSF grant DMS}\\Department of Mathematical Sciences\\Carnegie Mellon University\\Pittsburgh PA 15213}

\title{Multitrees in random graphs}
\maketitle
\begin{abstract}
Let $N=\binom{n}{2}$ and $s\geq 2$. Let $e_{i,j},\,i=1,2,\ldots,N,\,j=1,2,\ldots,s$ be $s$ independent permutations of the edges $E(K_n)$ of the complete graph $K_n$. A {\em MultiTree} is a set $I\subseteq [N]$ such that the edge sets $E_{I,j}$  induce spanning trees for $j=1,2,\ldots,s$. In this paper we study the following question: what is the smallest $m=m(n)$ such that w.h.p. $[m]$ contains a \MT. We prove a hitting time result for $s=2$ and an $O(n\log n)$ bound for $s\geq 3$.
\end{abstract}
\section{Introduction}
Let $N=\binom{n}{2}$ and $s\geq 2$. Let $e_{i,j},\,i=1,2,\ldots,N,\,j=1,2,\ldots,s$ be $s$ independent permutations of the edges $E(K_n)$ of the complete graph $K_n$. Let $\be_i=(e_{i,1},e_{i,2},\ldots,e_{i,s})$ and for $I\subseteq [N]$ let $E_{I,j}=\set{e_{i,j}:i\in I}$ for $j=1,2,\ldots,s$. A {\em MultiForest} is a set $I\subseteq [N]$ such that the edge sets $E_{I,j}$  induce forests for $j=1,2,\ldots,s$. A {\em MultiTree} is a \MF\ in which each forest is a spanning tree. In this paper we study the following question: what is the smallest $m=m(n)$ such that w.h.p. $[m]$ contains a \MT.

This is a particular case of the following more general question: given matroids $\cM_1,\cM_2,\ldots,\cM_s$ over a common ground set $E=\set{e_1,e_2,\ldots,e_M}$ let 
\[
\cI_k=\set{I\in \binom{[M]}{k}:\;\set{e_i,i\in I}\text{ is independent in }\cM_i,i=1,2,\ldots,s}.
\]
Then let $m^*=\max\set{k: \cI_k\neq \emptyset}$. Then we can ask what is the smallest $m=m(n)$ such that w.h.p. $[m]$ contains a member of $\cI_{m^*}$. In general this is a rather challenging question, mainly because the structure of randomly chosen matroids is not as well understood as the structure of random graphs.

There is at least one instance where we already have a precise answer to the above matroid question. We let $M=N$ and let $\cM_1$ be the graphic matroid of $K_n$. For $\cM_2$ we randomly color each edge $e\in E(K_n)$ uniformly with $c(e)\in C,|C|\geq n-1$ and $\cM_2$ is the partition matroid where a set $I\subseteq E(K_n)$ is independent if $e_1,e_2\in I$ implies that $c(e_1)\neq c(e_2)$. In more familiar terminology, $I$ is {\em rainbow colored}. This problem was solved in Frieze and McKay \cite{FM} where it was shown that w.h.p. $m^*$ is the smallest integer $m$ such that the graph induced by $e_1,e_2,\ldots,e_m$ is (i) connected and (ii) $|\set{c(e_i):i=1,2,\ldots,m}|\geq n-1$.

Going back to \MT's, we prove two theorems.
\begin{theorem}\label{th1}
 We have w.h.p. that $m^*=O(n\log n)$.
\end{theorem}
When $s=2$ we can use Edmond's theorem \cite{Ed} to prove the following: let  $\G_{j,m}=([n],E_{[m],j})$. 
\begin{theorem}\label{th2}
W.h.p. $m^*=\max\set{m_1,m_2}$ where for $j=1,2$, $m_j,=\min\set{m:\G_{j,m}\text{ is connected}}$.  
\end{theorem}
There is no actual need to restrict attention to matroid intersection. For example let $I$ be a {\em Multimatching} if the sets $E_{I,j},j=1,2,\ldots,s$ induce matchings and let $I$ be a {\em MultiPerfectMatching} if $|I|=\rdown{n/2}$ i.e. if the associated matchings are (near) perfect. 
\begin{theorem}\label{th3}
W.h.p. $[m]$ contains a MultiPerfectMatching if $m\geq Kn\log n$ for some absolute constant $K$.
\end{theorem}
One thing missing from this paper is what might be called MultiHamiltonCycle, where the edge sets $E_{I,j},j=1,2,\ldots,s$ induce Hamilton cycles. We have no results on this at present, but we conjecture that $m\geq Kn\log n$ should be enough for the existence of such a structure, w.h.p.
\section{Proof of Theorem \ref{th1}}
{\bf Phase 1:} In this phase we greedily add $s$-tples until we have a \MF\ of linear size. Consider the following construction: let $I_0=\emptyset$ and $k_0=0$. After $t$ steps we will have a \MF\ $I_t=\set{k_1=1,k_2,\ldots,k_t}$. Given $I_t$ we say that an $s$-tple $\be_k$ is {\em addable} to $I_t$ if $I_t\cup \set{k}$ is a \MF. Let $k_{t+1}=\min\{k>k_t:\be_k$ is addable to $I_t\}$. We let $F_{t,j}$ denote the forest induced by $\set{e_{k_i,j}:i=1,\ldots,t}$. We stop this greedy process after we have constructed $I_{m_0}$ where $m_0$ is defined below. 

To analyse this process, we need to understand the component structure of the forests $F_{t,j}$. Consider the ordinary graph process $\G_i,i=1,2,\ldots,N$. For $r\geq 1$, let $a_\ell=\min\set{r:\G_{r}\text{ has $n-\ell$ components}}$. The distribution of component sizes in $F_{\ell, j}$ will be the same as the distribution of component sizes in $\G_{a_\ell}$. This follows by induction on $\ell$. In all cases, we merge two components with probability proportional to the product of their sizes.

Recall next that if $c>1$ is a constant then w.h.p. the random graph $\G_{cn/2}$  has $\approx \k(c)n$ components and a unique giant component of size $\approx ng(c)$ where $\k(c),g(c)$ are known functions of $c$. For a proof of this, see for example Frieze and Karo\'nski \cite{FK}, Chapter 2.

Suppose now that we let $c_0=g^{-1}(1/2)$. Thus w.h.p. $G_{n,c_0n/2}$ contains a unique giant component of size $\approx n/2$. With regard to our greedy process, after examining some number of $s$-tples we will w.h.p. have constructed a multi-forest $I_{m_0}$ on $m_0\approx n(1-\k(c_0))$ $s$-tuples, where each individual forest $F_{m_0,j}$ (i) contains a giant tree of size $\approx n/2$ and (ii) has $n_0=n-m_0-2$ small components. The vertices of forest $F_{m_0,j}$ not in $T_j$ form a collection $S_j$ of small trees $T_{1,j},T_{2,j}\ldots,T_{n_0,j}$, each of size $O(\log n)$. 

We next consider as to how long we have to run this part of the process altogether. We first consider the time taken to get giant trees of size $\approx n/2$. We know that w.h.p. up until we have added $c_0n$ $s$-tples the probability that an $s$-tple \bof\ can be added to our forest is at least  $\g\approx \bfrac{3}{4}^s$. This is because (i) the larger the forest, the less likely a random edge can be added without creating a cycle and (ii) unless a random edge has both vertices in the giant, it is unlikely to create a cycle. This follows from the fact that w.h.p. each non-giant component is of size $O(\log n)$, in which case the the probability of choosing an edge with both vertices in the same small component is $O(n\times \bfrac{\log n}{n}^2)$. Thus w.h.p. it requires at most $\frac{2c_0n}{\g}$ iterations to produce a giant tree of size $n/2$.

{\bf Phase 2:} We now discuss how we can complete $I_{m_0}$ to a multi-tree. An $s$-tuple $\be_{u}$ will be {\em acceptable} if for each $j\in[s]$, the edge $e_{u,j}$ has one vertex $x_{u,j}\in A_j=[n]\setminus V(T_j)$ and the other $y_{u,j}\in T_j$. An acceptable $s$-tple defines an edge in a random $s$-uniform multi-partite hypergraph $H$ with edges in $A_1\times A_2\times \cdots\times A_s$. The vertices of $H$ are $A_1\sqcup A_2\sqcup\cdots\sqcup A_s$. 

We continue the process of adding acceptable $s$-tples until $H$ contains a set of edges $(x_{t,1},x_{t,2}\ldots,x_{t,s}),t\in K$ for some set $K$ of size $n_0$ that satisfies the following property: if $X_j=\set{x_{t,j}:t\in K}$ then $|X_j\cap V(T_{l,j})|=1$ for all $1\leq i\leq n_0,1\leq j\leq s$. This ensures that for each $j$ and each non-giant tree $T$ of $F_{m_0,j}$ that exactly one of the $n_0$ edges added to the $j$th forest joins $T$ and the giant $T_j$, thus creating a \MT. We call such matchings {\em \MT\ inducing}.

We next consider the number of random $s$-tples we need to generate before we have a \MT inducing matching in $H$ w.h.p. Suppose now that $T_{i,j}$ has $t_{i,j}$ vertices for $i=1,2,\ldots,n_0,j=1,2,\ldots,s$. We consider the hypergraph $\cH$ with vertex set $X$ equal to the {\em edges} of the complete $s$-partite hypergraph $\cA_s$ on $A_1\times A_2\times \cdots\times A_s$. An edge of $\cH$ corresponds to a \MT\ inducing matching of $\cA_s$. We will argue that w.h.p. $O(n\log n)$ randomly chosen {\em vertices} of $\cH$ contain an edge of $\cH$. To do this we will use a recent breakthrough result of Frankston, Kahn, Narayanan and Park \cite{FKNP2020}. For this we need a definition. For a set $S\subseteq X=V(\cH)$ we let $\upp{S}=\set{T:\;S\subseteq T\subseteq X}$ denote the subsets of $X$ that contain $S$. We say that $\cH$ is $\k$-spread if
\[
|\cH\cap \upp{S}|\leq \frac{|\cH|}{\k^{|S|}},\quad\forall S\subseteq X.
\]
The following theorem is from  \cite{FKNP2020}
\begin{theorem}\label{T2020}
Let $\cK$ be an $r$-uniform, $\k$-spread hypergraph and let $X=V(\cK)$. There is an absolute constant $C>0$ such that if 
\beq{mbound}{
m\geq\frac{(C\log r)|X|}{\k}
}
then w.h.p. $X_m$ contains an edge of $\cK$. Here w.h.p. assumes that $r\to\infty$.
\end{theorem}
To apply the lemma we prove
\begin{lemma}\label{spread}
W.h.p., $\cH$ is $\k$-spread, where $\k=(n_0/3)^{s-1}$.
\end{lemma}
\begin{proof}
We begin with the claim that
\beq{X}{
|\cH|=n_0!^{s-1}\prod_{i=1}^{n_0}\prod_{j=1}^{s}t_{i,j}.
}
We justify \eqref{X} as follows: if we fix a $j$ then there are $\prod_{i=1}^{n_0}t_{i,j}$ ways of choosing a single vertex from each $T_{i,j}$. After this, there are $n_0!$ ways of ordering these choices giving $\t_j$ choices altogether. We then multiply the $\t_j$ together to get the number choices for an edge ordered \MT\ inducing matching. We divide by $n_0!$ to remove the overcount due to ordering.

Suppose now that $S\subseteq X$ and $|S|=k$ and $\cH\cap \upp{S}\neq\emptyset$. Each element of $S$ is an $s$-tple. Let $S_j$ denote the $j$th component of the tples of $S$. Then 
\[
|\cH\cap \upp{S}|=(n_0-k)!^{s-1}\prod_{i\notin S_j}\prod_{j=1}^{s}t_{i,j}
\]
So, if $k<n_0$ then 
\[
\frac{|\cH\cap \upp{S}|}{|\cH|}\leq \bfrac{(n_0-k)!}{n_0!}^{s-1}\leq \brac{2\cdot\frac{(n_0-k)^{n_0-k}e^{n_0}(2\p n_0)^{1/2}}{n_0^{n_0}e^{n_0-k}(2\p(n_0-k))^{1/2}}}^{s-1} <\bfrac{2^{1/k}\exp\set{\frac{1}{2(n_0-k)}+\frac{k}{n_0}}}{n_0}^{k(s-1)},
\]
and the lemma follows.
\end{proof}
In the application of Theorem \ref{T2020} we have $r=sn_0$ and $|X|\approx (n/2)^s$. Applying the theorem we see that $C_1n\log n$ acceptable $s$-tples suffice to contain a \MT-inducing matching w.h.p. (Here we can take $C_1=(3\n_0)^{s-1}2^{-s}C$ where $\n_0=n/n_0=O(1)$ w.h.p.) An $s$-tple is acceptable with probability $\approx 2^{-2s}$ and so w.h.p. we need at most $2^{2s+1}Cn\log n$ $s$-tples overall before we obtain a \MT\ w.h.p. This completes the proof of Theorem \ref{th1}.
\section{Proof of Theorem \ref{th2}}
For the special case of $k=2$, we can use Edmond's analysis \cite{Ed} of the matroid intersection problem. Let $X_m=\set{(e_i,f_i):i=1,2,\ldots,m}$ be the set of pairs of random edges selected and let $\G_{1,m},\G_{2,m}$ be the two copies of $G_{n,m}$ induced by $X_m$. For $A\subseteq X_m$ let $\G_{1,m}(A)$ be the subgraph of $\G_{1,m}$ induced by the set of edges $e_i,i\in A$. Define $\G_{2,m}(A)$ similarly. We have to show that w.h.p. that for all $A\subseteq X_m$, we have 
\[
\k(A)=\k_1(A)+\k_2(X_m\setminus A)\leq n+1.
\]
 Here $\k_i(A),i=1,2$ denotes number of components in the graph $\G_{i,m}(A)$. 

Throughout this section: $N=\binom{n}2$ and 
\[
m=\tfrac12n(\log n+O(\log\log n))\text{ and }p=\frac{m}{N}.
\]

For $I\subseteq [m]$, we let $e(I)=\set{e_i:i\in I}$ and $f(I)=\set{f_i:i\in I}$ (with respect to $X_m=\set{(e_i,f_i):i=1,\dots,m}$). 

In the context of $G_{n,p}$, given a set of edges $A$, we let $V_m(A),V_p(A)$ be the set of vertices of $G_{n,m},G_{n,p}$ induced by $A$ and we let $v_m(A)=|V_m(A)|$ and $v_p(A)=|V_p(A)|$. Conversely, for a set of vertices $S$ let $E_m(S),E_p(S)$ denote the set of edges of $G_{n,m},G_{n,p}$ induced by $S$ and let $e_m(S)=|E_m(S),e_p(S)=|E_p(S)|$. For a set of vertices $S$ let $b_m(S)=e_m(S)/|S|,b_p(S)=e_p(S)/|S|$. 

\paragraph{Universal Parameters}
\beq{univ}{
\om=\log^{2/5}n; \ \e=\frac{1}{\om}; \ \th_0=\frac{5\log\log n}{\log n}; \ \s_0=\frac{10\om^2\log\log n}{\log n}; \ a_{max}=\frac{m}{2n\log n}\approx \frac{1}4; \ s_1=\frac{\log n}{10\log\log n}.
}
Define $\s(a)$ by 
\[
an\log n=(1+\e)\binom{\s(a) n}{2}p=(1+\e)\binom{\s(a) n}{2}\frac{m}{N},
\]
so that this is roughly the order of a subgraph expected to have $an\log n$ edges.
Then, 
\beq{s(a)}{
\s(a)=\bfrac{2a(1+\th)}{1+\e}^{1/2}\text{ where }|\th|\leq  \th_0.
}
We now bound $\k(A)=\k_1(A)+\k_2(X_m\setminus A)$ for $A\subseteq X$ with various ranges for 
\[
|A|=an\log n\leq |X_m\setminus A|.
\]
We begin each case analysis with a structural lemma. 
Let
\[
I_{conn}=[m_-,m_+]\text{ where }m_-=\tfrac12n(\log n-\log\log n),\,m_+=\tfrac12n(\log n+\log\log n).
\]

We consider the graph process $\cG=(G_m,m=0,1,\ldots,N)$ where as usual $G_{m+1}$ is obtained from $G_m$ by adding a random edge. We say that $\cG$ holds property $\cP$ {\em strongly} if w.h.p. $G_{m}\in\cP$ simultaneously for all $m\in \Ic$. We note that $G_m$ is distributed as $G_{n,m}$ and we when we refer to $G_p$ we mean $G_{n,p}$ for $p=m/N$, for some $m\in\Ic$.

\subsection{First Structural Lemmas}
We will assume from now on that $A$ induces components $C_1,C_2,\ldots,C_\ell$ in $\G_{1,m}$, where 
\[
1=|C_1|=\cdots=|C_k|<|C_{k+1}|\leq  \cdots \leq |C_\ell|.
\] 
\begin{lemma}\label{a}
The following hold strongly in $\cG$: in the statements, $p=m/N$ and $m\in \Ic$.
\begin{enumerate}[(a)]
\item If $|S|\geq \s_0n$ then $e_m(S)\in (1\pm\e)\binom{s}{2}p$.
\item If $|S|\leq n_0=\frac{n}{\log^3n}$ then $b_m(S)\leq 2$.
\item Let $i_0=\max\set{i:|C_i|\leq n_0}$. Then $|\ell-i_0|\leq \log^3n$.
\item $e_{1,m}(C_1)+\cdots+e_{1,m}(C_{i_0})\leq 2(|C_1|+\cdots+|C_{i_0}|)\leq 2n$. 
\end{enumerate}
\end{lemma}
As with most of the structural lemmas, the proof of Lemma \ref{a} is deferred to an appendix.

We break the possible range for $a$ into 3 intervals. We show for each individual range that the assumption $\k(A)>n+1$ leads w.h.p. to a contradiction.
\subsection{Case 1: $a_1=10^{-3}\leq a\leq a_{\max}$}
\begin{lemma}\label{aaa}
The following hold strongly in $\cG$: in the statements, $p=m/N$ and $m\in \Ic$ and $a_1\leq a\leq a_{\max}$.
\begin{enumerate}
\item If $|A|=an\log n$ then $v_m(A)\geq  \s(a)n$.
\item $|C_{k+1}|+\cdots+|C_\ell|\geq \bfrac{2a(1-\th_0)}{1+\e}^{1/2}n$. 
\item $i_0\leq \brac{1-\frac{(2a-3\th_0)^{1/2}}{1+\e}}n$.
\end{enumerate}
\end{lemma}
If $|A|=an\log n$ where $a_1\leq a\leq a_{max}$ and using that $m\in \Ic$,
\[
|X_m\setminus A|=m-an\log n=\brac{\frac12-a+O\bfrac{\log\log n}{\log n}}n\log n.
\] 
Applying Lemma \ref{a}(c) and Lemma \ref{aaa}(c) we see that
\begin{equation}
  \k(A)\leq \brac{\brac{1-\frac{(2a-3\th_0)^{1/2}}{1+\e}} +\brac{1-\frac{(1-2a-3\th_0)^{1/2}}{1+\e}}}n+2\log^3n.
\end{equation}
Using that
\[
(2a-3\th_0)^{1/2}=(2a)^{1/2}\brac{1-\frac{3\th_0}{2a}}^{1/2}\geq (2a)^{1/2}\brac{1-\frac{3\th_0}{4a}}\geq (2a)^{1/2}-\frac{2\th_0}{a^{1/2}},
\]
we have then that
\[
\k(A)\leq 2\log^3n+\brac{2+\frac{2\th_0}{a^{1/2}}+ \frac{2\th_0}{(1-2a)^{1/2}}-\frac{(2a)^{1/2}+(1-2a)^{1/2}}{1+\e}}n.
\]
But if $x=2a<1$ then
\[
x^{1/2}+(1-x)^{1/2}\geq x^{1/2}+1-\frac{x}{2}\geq 1+\frac{x^{1/2}}{2}.
\]
So,
\begin{align*}
\k(A)&\leq 2\log^3n+\brac{2+\frac{2\th_0}{a^{1/2}}+\frac{2\th_0}{(1-2a)^{1/2}}-\frac{1+a^{1/2}}{1+\e}}n\\
&\leq 2\log^3n+\brac{1+\frac{2\th_0}{a^{1/2}}+\frac{2\th_0}{(1-2a)^{1/2}}+ \frac{\e-a^{1/2}}{1+\e}}n\leq n.
\end{align*}
since $a^{1/2}\geq a_1^{1/2}\gg \max\set{\e,\tfrac{\log^3n}{n}}$ and $a_1\gg\th_0$.
\eoc{1}
\subsection{Another structural lemma}
\begin{lemma}\label{b}
The following hold strongly in $\cG$: in the statements, $p=m/N$.
\begin{enumerate}[(a)]
\item If $S$ induces a bridgeless subgraph which is not an induced cycle, then $|S|\geq s_1=\frac{\log n}{10\log\log n}$.
\item There are at most $n^{1/2}$ cycles of length at most $s_1$.
\end{enumerate}
\end{lemma}
In what follows, we will assume that $A$ maximises $\k(A)$ subject to $|A|\leq a_1n\log n$. Suppose also that $|A|$ is as small as possible subject to this maximisation. 
\begin{Remark}
If $C_i$ is not an isolated vertex, then we can assume that $C_i$ has no bridges. If $e_i,i\in A$ is a bridge of $\G_{1,m}(A)$ then replacing $A$ by $A\setminus\set{i}$ does not decrease $\k(A)$ and decreases $|A|$. 
\end{Remark}
We can therefore assume that
\beq{C1}{
1=|C_1|=\cdots=|C_k|<3\leq |C_{k+1}|\leq\cdots\leq |C_\ell|,
}
 where $C_{k+1},\ldots,C_\ell$ are bridgeless.

\begin{lemma}\label{d}
The following hold strongly in $\cG$: if $a_2:=3n^{-4/25}\leq a\leq a_1$ then $e_m(S:\bar{S})\geq 2an\log n$ for all $S,|S|\in [10an,n-10an]$.
\end{lemma}
\begin{Remark}\label{rem2}
If follows from Lemma \ref{d} that if $a_2\leq a\leq a_1$ then w.h.p. $\G_{2,m}(X_m\setminus A)$ contains a component of size $n-10an$.  
\end{Remark}

\subsection{Case 2: $a_2=3n^{-4/25}\leq a< a_1=10^{-3}$}
\begin{lemma}\label{e}
 The following hold strongly in $\cG$: if $a_2\leq a\leq a_1$ and $|S|\leq 12an$ then $b_m(S)< \frac{\log n}{12}$.
\end{lemma}
It follows from Lemma \ref{b}(a)(b) and Remark \ref{rem2} that if $a\geq a_2$ then with $k$ as in \eqref{C1}, 
\beq{f1}{
\k(A)\leq k+\frac{n-k}{s_1}+n^{1/2}+10an+1\leq\brac{n-(n-k)\brac{1-\frac{10\log\log n}{\log n}}}+11an.
}
{\bf Explanation:} In $\G_{1,m}$ there are $k$ isolated vertices plus at most $(n-k)/s_1$ bridgeless non-cycle components/large cycles plus at most $n^{1/2}$ small cycles. In $\G_{2,m}$ there is one giant component plus at most $10an$ vertices on small components.

Equation \eqref{f1} implies that if $\k(A)\geq n$ then $n-k\leq 12an$. Lemma \ref{e} gives us a contradiction in that w.h.p. $12an$ vertices do not induce $an\log n$ edges. 
\eoc{2}
\begin{lemma}\label{g}
The following holds strongly in $\cG_1$ and $\cG_2$ where we consider the two processes defined by $e_i,f_i,i=1,2,\ldots,N$:
\begin{enumerate}[(a)]
\item $\G_{2,m}$ contains at most $\log^{12}n$ vertices of degree at most 10.
\item The vertices of degree at most 10 in $\G_{2,m}$ are at distance at least 3 from each other.
\item If $f_i,i\in I,\,|I|\leq 10\log^{12}n$ are incident with vertices of degree at most 10 in $\G_{2,m}$ then $\set{e_i,i\in I}$ is not contained in any set of at most $s_1|I|$ vertices that induce a 2-edge-connected subgraph of $\G_{1,m}$.
\end{enumerate}
\end{lemma}

\subsection{Case 3: $0<a< a_2$}
\begin{lemma}\label{f}
The following holds strongly in $\cG$: if $|S|\leq n^{9/10}$ then $b_m(S)\leq 1+o(1)$.
\end{lemma}
Let $S^*=\bigcup_{j=k+1}^\ell C_j$ and $s^*=|S^*|$ and observe that $m^*:=|A|=e_{1,m}(S^*)$. We are left to consider the situation where we delete $m^*\leq m_1=a_2n\log n=3n^{21/25}\log n$ edges from $\G_{2,m}(X_m)$. Suppose that $\r(m^*)$ is the maximum number of components obtainable by deleting $m^*$ edges from $\G_{2,m}$. Since we can assume there are no bridges in $\G_{1,m}(A)$, any nontrivial component in $\G_{1,m}(A)$ has size at least 3 and we have that
\beq{kA}{
\k(A)\leq k+\frac{n-k}{3}+\r(m^*).
}
And
\beq{kB}{
m^*\approx n-k.
}
This is because w.h.p. the number of edges in a bridgeless component of size $3\leq s\leq n^{9/10}$ lies in $[s,(1+o(1))s]$ edges. The lower bound is true for all such sets and the upper bound follows from Lemma \ref{f}. 

Suppose that after removing $m^*$ edges from $\G_{2,m}$ we have components $K_1,K_2,\ldots,K_\r,\r=\r(m^*)$ where $|K_1|\leq |K_2|\leq \cdots\leq |K_\r|$. We then know from Remark \ref{rem2} that $|K_\r|\geq n-10a_2n$. We have from Lemma \ref{f} that $e_{2,m}(K_i)\approx |K_i|$ for $1\leq i\leq \r-1$. Let $\deg_2(v)$ denote the degree of vertex $v$ in $\G_{2,m}$ and $\deg_2(X)=\sum_{x\in X}\deg_2(x)$. For a fixed $i\in[\r-1]$, the set $f_j,j\in A$ must contain  $\deg_2(K_i)-2e_{2,m}(K_i)\geq \deg_2(K_i)-(2+o(1))|K_i|$ edges with exactly one end in $K_i$. Thus, if there are $\r_1$ single vertex components then
\[
m^*\geq \frac 1 2 \sum_{i=1}^{\r-1}(\deg_2(K_i)-(2+o(1))|K_i|)=\frac 1 2 \sum_{i=1}^{\r_1}\deg_2(K_i)+\frac 1 2\sum_{i=\r_1+1}^{\r-1}(\deg_2(K_i)-(2+o(1))|K_i|).
\]
Define $\rho_0\leq \rho_1$ such that, among the single-vertex components, we have $\deg_2(K_i)\leq 10$ for $1\leq i\leq \r_0$ and $\deg_2(K_1)\geq 11$ for $\r_0+1\leq i\leq \r_1$. It follows from Lemma \ref{g}(b) that at least half of the vertices in any non-trivial $K_i,i>\r_1$ have degree at least 11. This is because the neighborhoods of the low degree vertices are disjoint and non-empty. So, $\deg_2(K_i)\geq 8|K_i|/2$ for all $i>\r_0$. Thus,
\[
m^*\geq \frac12\brac{\r_0+\frac{8(\r-1-\r_0)}{2}}=\frac{8(\r-1)}4-\frac{8\r_0}4.
\]
The initial factor of $\tfrac12$ arises because the same edge might be counted twice, once for each of the $K_i$ that it is incident with. 

It follows from Lemma \ref{g}(c) and \eqref{C1} that $\r_0\leq m^*/s_1$. (The edges deleted from $\G_{2,m}$ correspond to edges of bridgeless components in $\G_{1,m}$. And then Lemma \ref{g}(c)  implies that each edge that was deleted in $\G_{2,m}$ to create $K_i,i\leq \r_0$ can be ``charged'' to $s_1$ distinct edges of $\G_{1,m}$.) So, $\r-1\leq (4+o(1))m^*/8$. In which case \eqref{kA} and \eqref{kB} imply that
\[
\k(A)\leq k+(n-k)\brac{\frac{1}{s_1}+\frac{1}{3}+\frac{4+o(1)}{8}}+1\leq n+1.
\]
\eoc{3}

When $a=0$ we rely on the connectivity of both $\G_{1,m},\G_{2,m}$.
\subsection{Hitting time}\label{hit}
The essence of the above argument is that if $\G_{1,m},\G_{2,m}$ are both connected and satisfy the conditions of Lemmas \ref{a} -- \ref{f} then there is a \MT. It is well known that the hitting time $m_c$ for connectivity and minimum degree at least one satisfies $m_c\in[m_-,m_+]$ w.h.p. Thus to verify the claim for a hitting time, we only have to show that Lemmas \ref{a} -- \ref{f} are valid for $G_{n,m},m\in [m_-,m_+]$. The reader will observe that we have been careful to do precisely this.
\section{Proof of Theorem \ref{th3}}
We first consider a multi-partite version where the edges $e_{i,s}$ are drawn from disjoint copies of the edges of the complete bipartite graph $K_{n,n}$. In this case, \MP's are in 1-1 correspondence with perfect matchings in the complete $2s$-uniform multi-partite hypergraph with edges set $[n]^{2s}$. As such it is known that a random set of $Kn\log n$ edges is sufficient for a perfect matching w.h.p. It is tempting to take $K=1$ and refer to Kahn \cite{K0}, \cite{K1}. On the other hand, one can legitimately cite \cite{JKV} or \cite{BF} and get some constant $K$.

With the above case in hand, one gets Theorem \ref{th3} by partitioning $[n]$ randomly into $2s$ parts of $V_1,V_2,\ldots,V_{2s}$ of size $\rdown{n/2s}$ and discarding at most $2s-1$ vertices. Then we only consider those $\be_i=(x_1,x_2,\ldots,x_{2s})$ and appeal to the above multi-partite version.

If we want to assume that $n$ is even and consider $s$ perfect matchings in $K_n$ then we can partition $[n]$ into two sets $A,B$ of size $n/2$ and only consider those $\be_i$ where all the $e_{i,j}$ have one end in $A$ and the other in $B$. We then have to inflate the $K$ of the first paragraph by at most $2^s$. This idea can be extended to deal with tree factors as in {\L}uczak and Ruci\'nski \cite{LR}.
\section{Final Remarks}
We have proved some threshold results for the intersections of cycle matroids. It would be of interest to extend this to other classes of matroid, e.g. binary matroids. There is also the analogous problem with respect to Hamilton cycles. This seems to be more difficult.

\appendix
\section{Proof of Structural Lemmas}
\paragraph{Universal Parameters}
\[
\om=\log^{2/5}n; \quad \e=\frac{1}{\om}; \quad \th_0=\frac{5\log\log n}{\log n}; \quad \s_0=\frac{10\om^2\log\log n}{\log n}; \quad a_{max}=\frac{m}{2n\log n}\approx \frac{1}4.
\]
\paragraph{Lemma \ref{a}}
{\it 
The following hold strongly in $\cG$: in the statements, $p=m/N$ and $m\in \Ic$.
\begin{enumerate}[(a)]
\item If $|S|\geq \s_0n$ then $e_m(S)\in (1\pm\e)\binom{s}{2}p$.
\item If $|S|\leq n_0=\frac{n}{\log^3n}$ then $b_m(S)\leq 2$.
\item Let $i_0=\max\set{i:|C_i|\leq n_0}$. Then $|\ell-i_0|\leq \log^3n$.
\item $e_{1,m}(C_1)+\cdots+e_{1,m}(C_{i_0})\leq 2(|C_1|+\cdots+|C_{i_0}|)\leq 2n$. 
\end{enumerate}
}
\begin{proof}
(a) It follows from the Chernoff bounds that in $G_{n,p}$
\mult{1}{
\Pr\brac{\exists S,|S|=s\geq \s_0n: e_p(S)\notin (1\pm\e)\binom{s}{2}p}\leq 2\sum_{s=\s_0n}^n\binom{n}{s} \exp\set{-\frac{\e^2s(s-1)p}{6}}\leq\\
 2\sum_{s=\s_0n}^n\brac{\frac{ne}{s}\cdot \exp\set{-\frac{\e^2s\log n}{7n}}}^s=2 \sum_{s=\s_0n}^n\brac{\frac{ne}{s}\cdot \exp\set{-\frac{s\log n}{7\om^2n}}}^s=o(n^{-3}).
}
Now for any graph property $\cP$ we have 
\beq{transfer}{
\Pr(G_{n,m}\in\cP)\leq 10m^{1/2}\Pr(G_{n,p}\in\cP).}
There are many possible references for this result, see for example Lemma 1.2 of \cite{FK}. We will generally use \cite{FK} for references.

The claim for all $m\in I_{conn}$ then follows directly from \eqref{1} and \eqref{transfer}. The probability there exists an $S$ being $O(n^{-3}\times (n\log n)^{1/2}\times n\log\log n)$.

(b)
\mult{Fu}{
\Pr\brac{\exists S,|S|=s: b_p(S)\geq b}\leq \binom{n}{s} \binom{\binom{s}{2}}{b s}p^{b s}\leq \brac{\frac{ne}{s}\cdot \bfrac{s^2e^{1+o(1)}\log n}{2b sn}^{b}}^s\\
 =\brac{\frac{ne}{s}\bfrac{se^{1+o(1)}\log n}{2bn}^{b}}^s =\brac{\bfrac{se^{1+o(1)}\log n}{2bn}^{b-1}\cdot \frac{e^{2+o(1)}\log n}{2b}}^s.
}
It follows that 
\beq{aa}{
\Pr\brac{\exists S,|S|=s\leq n_0: b_p(S)\geq 2}\leq \sum_{s=5}^{n_0}\bfrac{e^{3+o(1)}}{8\log n}^s=o(1).
}
Now the event $\set{\exists S,|S|=s\leq n_0: b_p(S)\geq 2}$ is monotone increasing. For monotone increasing events $\cP$, \eqref{transfer} can be strengthened to 
\beq{transferx}{
\Pr(G_{n,m}\in\cP)\leq 3\Pr(G_{n,p}\in\cP).}
See for example Lemma 1.3 of \cite{FK}. Note also that we need only to prove this for $m=m_+$. In which case, \eqref{aa} also implies (b).

(c) This is obvious.

(d) This follows from (b).
\end{proof}
\paragraph{Lemma \ref{aaa}}
{\it 
The following hold strongly in $\cG$: in the statements, $p=m/N$ and $m\in \Ic$ and $a_1=10^{-3}\leq a\leq a_{\max}\approx \tfrac14$.
\begin{enumerate}[(a)]
\item If $|A|=an\log n$ then $v_m(A)\geq  \s(a)n$.
\item $|C_{k+1}|+\cdots+|C_\ell|\geq \bfrac{2a(1-\th_0)}{1+\e}^{1/2}n$. 
\item $i_0\leq \brac{1-\frac{(2a-3\th_0)^{1/2}}{1+\e}}n$.
\end{enumerate}
}
\begin{proof}
(a)  Since $a\geq a_1$ we have that $\s(a)\geq \s(a_1)$. We claim next that w.h.p. $V_m(A)$ is not the union of components. If $s=v_m(A)<\s(a)n$ then in $G_{n,p}$ we can bound this probability by
\mults{
\binom{n}{s}\binom{\binom{s}{2}}{an\log n}p^{an\log n}(1-p)^{s(n-s)} = \binom{n}{s}\binom{\binom{s}{2}}{(1+\e)\binom{\s(a)n}{2}p}p^{an\log n}(1-p)^{s(n-s)} \leq\\
 \bfrac{ne}{s}^s\bfrac{s^2}{(1+\e)(\s(a)n)^2}^{an\log n}e^{-snp(1-\s(a))}\leq \bfrac{en^{\s(a)}}{s}^s(1+\e)^{-a_1n\log n}\leq e^{-n}.
}
We can add vertices to create $B\supseteq V_m(A)$ with $|B|=\s(a)n$. Because $V_m(a)$ is not the union of components, we can assume that $e_m(B)>|A|$. We then see that w.h.p. $|A|<e_m(B)\leq an\log n$, contradiction. (The second inequality follows from Lemma \ref{a}(a) and the definition of $\s$.)

(b) This follows from \eqref{s(a)} and (a). We remind the reader that $|A|=an\log n$ and $a\geq a_1$ in this case and that $|C_{k+1}|+\cdots+|C_\ell|$ is the number of vertices in the subgraph of $\G_{1,m}$ induced by $A$.

(c) It follows from Lemma \ref{a}(a) that
\[
|C_{i_0+1}|+\cdots+|C_\ell|\geq \s\bfrac{an\log n-2i_0}{n\log n}\geq \s\brac{a-\frac{2}{\log n}} \geq \bfrac{2a-3\th_0}{1+\e}^{1/2}n.
\]
So,
\beq{k}{
i_0\leq |C_1|+\cdots+|C_{i_0}|\leq \brac{1-\frac{(2a-3\th_0)^{1/2}}{1+\e}}n.
}
\end{proof}
\paragraph{Lemma \ref{b}}
{\it
The following hold strongly in $\cG$: in the statements, $p=m/N$.
\begin{enumerate}[(a)]
\item If $S$ induces a bridgeless subgraph which is not an induced cycle, then $|S|\geq s_1=\frac{\log n}{10\log\log n}$.
\item There are at most $n^{1/2}$ cycles of length at most $s_1$.
\end{enumerate}
}
\begin{proof}
(a) A bridgeless graph is either a cycle or has $s$ vertices and at least $s+1$ edges. But then, in $G_{n,p}$,
\mult{sigma1}{
\Pr(\exists S:|S|\leq 2s_1,e_p(S)\geq |S|+1)\leq \sum_{s=4}^{2s_1}\binom{n}{s}\binom{\binom{s}{2}}{s+1}\bfrac{e^{o(1)}\log n}{n}^{s+1} \leq\\ 
\sum_{s=4}^{2s_1}\bfrac{ne}{s}^s\bfrac{se}{2}^{s+1}\bfrac{e^{o(1)}\log n}{n}^{s+1} =\sum_{s=4}^{2s_1}\frac{2\log n}{n}\bfrac{e^{2+o(1)}\log n}{2}^s=o(1).
}
Having such a set $S$ is a monotone increasing property and so we obtain the needed result from \eqref{transferx}. Again, we only need verify the property from $m=m_+$. (We use $2s_1$ in place of $s_1$ for use in (b).)

(b) If two small cycles share a vertex then there is a set $S$ of size at most $2s_1$ that contains at least $|S|+1$ edges. This was ruled out in the analysis of (b). So, we can count {\em selfish} small cycles, i.e. those that do not share vertices with other small cycles. Let $\n_0=n^{1/2}$ and let $s$ be a positive integer. Then 
\beq{choose}{
\Pr(Z\geq \n_0)=\Pr\brac{\binom{Z}{s}\geq \binom{\n_0}{s}} \leq \frac{\E\brac{\binom{Z}{s}}}{\binom{\n_0}{s}}
}
Now
\begin{align*}
\E\brac{\binom{Z}{s}}&=\sum_{3\leq \ell_1,\ell_2,\ldots,\ell_s\leq s_1} \binom{n}{\ell_1,\ell_2,\ldots,\ell_s,n-\ell_1-\cdots-\ell_s} \prod_{i=1}^s\frac{(\ell_i-1)!}{2}p^{\ell_i}\\
&\leq \sum_{3\leq \ell_1,\ell_2,\ldots,\ell_s\leq s_1}\prod_{i=1}^s\frac{e^{o(1)}\log n}{2\ell_i}
\leq \brac{\sum_{3\leq \ell\leq s_1}\frac{e^{o(1)}\log n}{2\ell}}^s\\
&\leq n^{s/9}.
\end{align*}
Going back to \eqref{choose} we see that
\[
\Pr(Z\geq \n_0)\leq \frac{n^{s/9}s^s}{\n_0^s}=o(n^{-3})
\]
if we take $s=30$. Thus the probability the claim fails for any $m\in I_{conn}$ is at most $O(m\log\log n\times m^{1/2}\times n^{-3})=o(1)$.
\end{proof}
\paragraph{Lemma \ref{d}}
{\it
 The following hold strongly in $\cG$: if $3n^{-4/25}=a_2\leq a\leq a_1=10^{-3}$ then $e_m(S:\bar{S})\geq 2an\log n$ for all $S,|S|\in [10an,n-10an]$.
}
\begin{proof}
We only have to prove this for $m=m_-$. First observe that if $10an\leq s\leq n/2$ then 
\[
\frac{s(n-s)p}{an\log n}=\frac{e^{o(1)}s(n-s)}{an^2}\geq \frac{e^{o(1)}s}{2an}\geq 5-o(1).
\]
It follows from Chernoff bounds, that in $G_{n,p}$ we have
\mults{
\Pr(\exists S,|S|\in [10an,n-10an]:e_p(S,\bar{S})\leq 2an\log n)\leq 2\sum_{s=10an}^{n/2}\binom{n}{s} e^{-s((3-o(1))/5)^2\log n/2}\leq\\ 2\sum_{s=10an}^{n/2}\bfrac{ne^{1-4\log n/25}}{s}^{s}\leq 2\sum_{s=10an}^{n/2}\bfrac{e}{3}^s=o(n^{-2}).
}
Now use \eqref{transferx}.
\end{proof}
\paragraph{Lemma \ref{e}}
{\it
 The following hold strongly in $\cG$: if $a_2\leq a\leq a_1$ and $|S|\leq 12an$ then $b_m(S)\leq \frac{\log n}{12}$.
}
\begin{proof}
We only have to prove this for $m=m_+$.  Applying \eqref{Fu}, we have that in $G_{n,p}$,
\mults{
\Pr\brac{\exists S,|S|=s\leq 12an: b_p(S)\geq \frac{\log n}{12}}\leq \sum_{s=5}^{12an}\brac{\bfrac{12e^{1+o(1)}a\log n}{2\log n/12}^{\log n/12-1}\cdot \frac{12e^{1+o(1)}\log n}{2\log n}}^s\\
\leq\sum_{s=5}^{12an}\brac{(73ea)^{\log n/12-1}\cdot 20}^s=o(n^{-2}).
}
The property in question is monotone decreasing. We can use \eqref{transferx}.
\end{proof}
\paragraph{Lemma \ref{g}}
{\it
The following holds strongly in $\cG_1$ and $\cG_2$ where we consider the two processes defined by $e_i,f_i,i=1,2,\ldots,N$:
\begin{enumerate}[(a)]
\item $G_{2,m}$ contains at most $\log^{12}n$ vertices of degree at most 10.
\item The vertices of degree at most 10 in $\G_{2,m}$ are at distance at least 3 from each other.
\item If $f_i,i\in I,\,|I|\leq 10\log^{12}n$ are incident with vertices of degree at most 10 in $\G_{2,m}$ then $\set{e_i,i\in I}$ is not contained in any set of at most $s_1|I|$ vertices that induce a bridge free connected subgraph of $\G_{1,m}$.
\end{enumerate}
}
\begin{proof}
(a) We only have to prove this for $m=m_-$. Let $S$ denote the set of vertices of degree at most 10 in $G_{m_-}$. If $p=p_-=m_-/N$ then in $G_{n,p}$ the expected size of $S$ can be bounded by
\[
n\sum_{i=0}^{10}\binom{n}{i}p^i(1-p)^{n-i}\leq 2n\sum_{i=0}^{10}\log^in\times \frac{\log n}{n}.
\]
The bound on the size of $S$ now follows from the Markov inequality. A bound on the size of $S$ is a monotone increasing property and so we can translate this error bound to $G_{n,m_-}$ via \eqref{transferx}. 

(b) Let $S$ denote the set of  vertices of degree at most 10 in $G_{n,p},p=p_-$ . The probability that there is a path of length at most two between two vertices in $S$ is at most
\[
\binom{n}{2}(p+np^2)\brac{\sum_{i=0}^{10}\binom{n-3}{i}p^i(1-p)^{n-3-i}}^2=O\bfrac{\log^{12}n}{n^{1-o(1)}}.
\]
Given (a), the probability that the $m_+-m_-$ additional edges in $G_{m_+}-G_{m_-}$ add an edge between two vertices in $S$ can be bounded by $O\brac{n\log\log n\times \log^{24}n/\binom{n}{2}}=o(1)$.

Given (a), the probability that the $m_+-m_-$ additional edges add an edge between $S$ and a neighbor of $S$ can be bounded by $O\brac{n\log\log n\times \log^{12}n/\binom{n}{2}}=o(1)$.

Given (a), the probability that the $m_+-m_-$ additional edges join a pair of vertices in $S$ to a $G_{m_-}$ non-neighbor can be bounded by $O\brac{(n\log\log n)^2\times \log^{24}n\times n/\binom{n}{2}^2}=o(1)$.

(c) Given the bound on the number of low degree vertices in (a), the probability that there exists a cycle of length $s$ in $\G_{1,m},m=m_+$ containing $t$ edges $e_i$ for which $f_i$ is incident with a vertex of degree at most 10 in $\G_{2,m_-}$ is at most
\beq{last}{
n^s\binom{s}{t}\frac{\binom{N-s}{m-s}}{\binom{N}{m}}\bfrac{10\log^{12}n}{n}^t \leq (2\log n)^s \bfrac{10\log^{12}n}{n}^t.
}
We must sum the RHS of \eqref{last} for $1\leq t\leq 10\log^{12}n$ and $3\leq s\leq s_1t$. Observing that $\log^{s_1+10}n=n^{1/10+o(1)}$, we see that this sum is $o(1)$. (Recall that $s_1=\frac{\log n}{10\log\log n}$.) If there is a cycle that contradicts (c) in the process then this cycle will occur in $G_{m_+}$ and the offending $f_i$ will be incident with low degree vertices in $G_{m_-}$.

Now consider bridge free connected sets. The probability that there is a set of size $s$ with $t$ edges $e_i$ of the required sort can be bounded by
\mults{
\binom{n}{s}\binom{\binom{s}{2}}{s+1}\frac{\binom{N-s-1}{m-s-1}}{\binom{N}{m}} \binom{s+1}{t}\bfrac{10\log^{12}n}{n}^t\leq \frac{se}{n}\brac{\frac{ne}{s}\cdot\frac{se}{2}\cdot\frac{m}{N}\cdot2}^{s+1} \bfrac{10\log^{12}n}{n}^t\leq \\
(e^{2+o(1)}\log n)^s\bfrac{10\log^{12}n}{n}^t.
}
We finish the argument as we did for cycles.
\end{proof}
\paragraph{Lemma \ref{f}}
{\it The following holds strongly in $\cG$: if $|S|\leq n^{9/10}$ then $b(S)\leq 1+10\th_0$.}
\begin{proof}
We only have to prove this for $m=m_+$. Applying \eqref{Fu}, we have that in $G_{n,p}$,
\mults{
\Pr\brac{\exists S,|S|=s\leq n^{9/10}: b_p(S)\geq 1+\frac{50\log\log n}{\log n}}\leq\\ \sum_{s=4}^{n^{9/10}} \brac{\bfrac{e^{1+o(1)}\log n}{2n^{1/10}}^{50\log\log n/\log n}\cdot e^{2+o(1)}\log n}^s \leq \sum_{s=4}^{n^{9/10}} \bfrac{e^5}{\log^{3}n}^s=o(1).
}
The property in question is monotone increasing and so we can apply \eqref{transferx}.
\end{proof}

\end{document}